\documentclass[11pt]{amsart}
\usepackage{amsmath}
  \usepackage{manfnt}
\usepackage{tikz}
\usepackage{tikz-cd}
\usepackage{amsfonts}
\usepackage{amsthm}
\usepackage{amssymb}
\usepackage{mathrsfs}
\usepackage{a4wide}
\usepackage{hyperref}
\usepackage{stmaryrd}

\def\Vg{V\kern-0.15em{g}}
\def\DD{\mathbf{D}}
\def\univ{\mathrm{univ}}

\def\Sw{\widetilde{S}}

\def\tr{\mathrm{tr}}
\def\qq{l}
\def\Pbar{\overline{P}}
\def\kbar{\overline{k}}

\def\m{\mathfrak{m}}
\def\T{\mathbf{T}}
\def\Z{\mathbf{Z}}
\def\OL{\mathcal{O}}
\def\Q{\mathbf{Q}}

\def\C{\mathbf{C}}
\def\GL{\mathrm{GL}}
\def\Frob{\mathrm{Frob}}
\def\eps{\epsilon}
\def\Gal{\mathrm{Gal}}
\def\Qbar{\overline{\Q}}
\def\rhobar{\overline{\rho}}

\DeclareMathOperator\Spec{Spec}
\DeclareMathOperator\End{End}

\newtheorem{theorem}{Theorem}[section]
\newtheorem{df}[theorem]{Definition}
\newtheorem{facts}[theorem]{Fact}
\newtheorem{lemma}[theorem]{Lemma}
\newtheorem{prop}[theorem]{Proposition}
\newtheorem{remark}[theorem]{Remark}

\def\wT{\widetilde{\T}}
\def\wP{\widetilde{P}}
\def\Ddag{\widetilde{D}^{\dagger}}
\def\Rd{R^{\dagger}}
\def\Rtd{\Rtw^{\dagger}}
\def\Rtw{\widetilde{R}}

\def\Tbar{\overline{T}}
\def\Dbar{\overline{D}}
\def\Pbar{\overline{P}}
\def\alphabar{\overline{\alpha}}
\def\betabar{\overline{\beta}}

\renewcommand{\ell}{l}

\thanks{F.C. \ was supported in part by NSF Grant  DMS-1701703.}

\begin{document}

\title{Pseudo-representations of Weight One are unramified}
\author{Frank Calegari and Joel Specter}
\begin{abstract}
We prove that the determinant (pseudo-representation) associated to
the Hecke algebra of Katz modular forms of weight one and level prime to~$p$
is unramified at~$p$.
\end{abstract}
\maketitle
{\footnotesize
\tableofcontents
}

\section{Introduction}
Let~$p$ be prime, and let~$N \ge 5$ be prime to~$p$. Let~$\OL$ be the ring of integers in a finite extension~$K$ of~$\Q_p$ with uniformizer~$\varpi$.
Let~$X_1(N)$  be the modular curve considered as a smooth proper curve over~$\Spec(\OL)$, and let~$\omega$ be the pushforward
of the relative dualizing sheaf along the universal elliptic curve. 
The coherent cohomology group~$H^0(X_1(N),\omega)$ may be identified with
the space of modular forms of weight one with coefficients in~$\OL$.
For general~$m$, one knows that the map:
$$H^0(X_1(N),\omega) \rightarrow H^0(X_1(N),\omega/\varpi^m)$$
need not be surjective. This was first observed by Mestre for~$N = 1429$ and~$p = 2$, (see~\cite[Appendix~A]{Bas}), and many examples for larger~$p$ have
been subsequently computed by Buzzard and Schaeffer~\cite{SmallBuzz,Schaeffer}.
In particular, if~$\T$ denotes the subring of
$$\End_{\OL} \lim_{\rightarrow} H^0(X_1(N),\omega/\varpi^m) = \End_{\OL}  H^0(X_1(N), \omega \otimes K/\OL),$$
generated by Hecke operators~$T_{\ell}$ and~$\langle \ell \rangle$ for~$(\ell,N) = 1$, then~$\T$ may be bigger than the classical Hecke algebra acting 
on the space~$H^0(X_1(N),\omega \otimes \C)$ of classical modular forms of weight one.
Let ~$G_{\Q}$ be the absolute Galois group of ~$\Q$. 
Let~$G_{\Q,N}$ be the absolute Galois group of the maximal extension of~$\Q$
unramified outside~$N \infty$. Our main theorem is as follows:

\begin{theorem} \label{theorem:main}
Let~$\T \subset \End_{\OL} H^0(X_1(N), \omega \otimes K/\OL)$ denote the algebra
generated by Hecke operators~$T_{\ell}$ and~$\langle \ell \rangle$ for all~$\ell$ prime to~$N$. There is a degree~$d = 2$ 
determinant\footnote{a notion of pseudo-representation
which works in all characteristics, see~\S\ref{section:pseudo}.}:
$$\DD: \T[G_{\Q}] \rightarrow \T, \qquad P(\DD,\sigma) = X^2 - T(\sigma) X + D(\sigma),$$
which is unramified outside ~$N\infty$ --- equivalently, which factors through~~$\T[G_{\Q,N}]$  --- such that for all primes ~$\ell \nmid N$, including~$\ell = p$, one has~
 ~$$T(\Frob_{\ell}) = T_{\ell} \ \text{and} \ D(\Frob_{\ell}) = \langle \ell \rangle.$$
 \end{theorem}

\medskip

The ring~$\T$ is a finite~$\OL$-algebra and  is moreover a semi-local ring, and thus is a direct sum~$\bigoplus \T_{\m}$ of its completions at maximal ideals~$\m$. 
For each maximal ideal~$\m$ of~$\T$, 
the residual determinant~$\Pbar: \OL[G_{\Q}] \rightarrow \T_{\m}/\m = k$
arises from to a semi-simple Galois representation~$\rhobar$ over~$\kbar$ (Theorem~A of~\cite{Chenevier}). If this representation
is irreducible, then~$P$ itself also arises from a genuine representation, which, by
a theorem of Carayol~\cite{carayol1994formes}, takes values in~$\T_{\m}$. It follows from Theorem~\ref{theorem:main}
that the corresponding representation
$$\rho: G_{\Q} \rightarrow  \GL_2(\T_{\m})$$
is unramified at~$p$.
For~$p > 2$, this is a consequence of Theorem~3.11 of~\cite{CG}. Hence the main
interest of  this result is to residually reducible representations.  However,
the result is new even for absolutely irreducible representations when~$p = 2$
 (although there are significant partial results by Wiese~\cite{Wiese}).
Although the proof of Theorem~\ref{theorem:main} is similar
to that of Theorem~3.11 of~\cite{CG}, it is more direct, and does not rely on any explicit analysis of the ordinary deformation rings of Snowden~\cite{Snowden}.
Hence this paper can  also be seen as providing a simplification of the proof of Theorem~3.11 of \emph{ibid}.
(See also Remark~\ref{rem:original}).

\medskip

The existence of the determinant without any condition at~$p$ is an easy consequence of the corresponding result in higher weight: first consider the action of~$\T$ on~$H^0(X_1(N), \omega/\varpi^m)$
and then multiply by  a suitable power of the Hasse invariant which is Hecke equivariant. Hence the main content of this theorem is that the determinant is unramified at~$p$.

\section{Determinants}

\label{section:pseudo}

In this paper, we will use  the term ``pseudo-representation'' as a catch-all to refer to various types of generalized representations. The first pseudo-representations were introduced by Wiles~\cite{WilesOrdinary}  for 2-dimensional representations;  these were later generalized to any dimension by Taylor~\cite{MR1115109}. Following Roquier~\cite{MR1378546}, we will call Taylor-style pseudo-representations
 ``pseudo-characters,'' because of their resemblance to the trace of a representation. In this paper, we will mainly consider the pseudo-representations of Chenevier~\cite{Chenevier} called ``determinants.'' These are more general and flexible than pseudo-characters, and in particular allow us to treat the case where~$p = d = 2.$ We shall only be concerned with determinants of degree~$d = 2$.

We begin by recalling the notion of a determinant ~\cite[pg.~223]{Chenevier}. 
Let ~$G$ be a group and ~$A$ be a ring. Let ~$d$ be a positive integer. If ~$M$ is a free, rank-$d$
$A$-module equipped with a linear ~$G$ action, then one may consider the 
family of characteristic polynomials associated to the elements of ~$A[G]$ acting on ~$M.$
This family of polynomials is highly interdependent, and is a robust invariant of the representation ~$M$.
Informally, a degree ~$d$ determinant is a pseudo-representation containing the information of a family of polynomials which satisfies the collection of common relations shared by all families of degree ~$d$ characteristic polynomials. If~$B$ is an~$A$-algebra, one can extend the action of~$A[G]$ on~$M$ to an action of ~$B[G]$ on ~${M\otimes_A B},$ and also obtain
corresponding characteristic polynomials over~$B$ for elements in~$B[G]$. Chenevier's definition of a determinant follows from the following two insights. First, the data of the characteristic polynomials for elements in ~$B[G]$ as one ranges over all ~$A$-algebras ~$B$ is equivalent to that of the literal determinants of the elements of ~$B[G]$ acting on  ~${M\otimes_A B}$ as one ranges over all ~$A$-algebras ~$B$: the characteristic polynomial of an element ~$m \in B[G]$ is, by definition, the determinant of the endomorphism ~$X-m$ acting on ~${M \otimes_{A} B[X].}$ Second, relations in families of characteristic polynomials arise via compatibilities of the determinant map. The literal determinants of the elements of~$B[G]$ acting on ~$M \otimes_A B$ can
be organized as a series of set theoretic maps ~$\det:B[G] \rightarrow B,$
one for each ~$A$-algebra ~$B,$ which satisfy the following
compatibilities:

\begin{enumerate}
\item the maps ~$\det$ are natural in ~$B,$
\item ~$\det(1) = 1$ and the element ~$\det(xy) = \det(x)\det(y)$ for all ~$x,y
\in B[G],$
\item and ~$\det(bx) = b^d \det(x),$ where ~$b\in B$ and ~$d$ is equal to the
rank of ~$M.$
\end{enumerate}

\noindent A \textit{determinant} is simply a family of maps which are
compatible in these three ways.

\begin{df} Let ~$A$ be a ring\footnote{All rings considered in this note
will carry a Hausdorff topology, and, with the exception of group rings,
will be commutative. Our terminology will suppress these topological and
algebraic considerations. We use the terms \textit{module} and
\textit{algebra} to denote a Hausdorff topological module and a
commutative, Hausdorff topological algebra, respectively. }, ~$G$ be a
topological group, and
$d$ be a positive integer. A degree ~$d$ determinant is a continuous
$A$-valued polynomial law\footnote{An ~$A$-valued polynomial law between
two ~$A$-modules ~$M$ and
$N$ is by definition a natural transformation ~$N\otimes_{A}B \rightarrow
M
\otimes_A B$ on the category of commutative ~$A$-algebras ~$B.$ A
polynomial
law is called multiplicative if ~${\DD}(1) = 1$ and ~${\DD}(xy) = {\DD}(x){\DD}(y)$ for
all
$x,y \in A[G]\otimes B,$ and is called homogeneous of degree ~$d,$ if
${\DD}(xb) = b^d{\DD}(x)$ for all ~$x \in A[G]\otimes B$ and ~$b \in B.$ A
polynomial law is called continuous if its characteristic polynomial map
on ~$G$ given by ~$g \mapsto P({\DD},g)$ is continuous.} ~${\DD}: A[G] \rightarrow
A,$ which is multiplicative
and homogeneous of degree ~$d.$ If ~$B$ is an ~$A$-algebra and ~$m \in
B[G],$ we call ~$P({{\DD}},m)(X)
:= {\DD}(X-m)\in B[X]$ the characteristic polynomial of ~$m.$
\end{df}

Given a determinant~${\DD}:A[G] \rightarrow A$ and an~$A$-algebra~$B,$ the restriction of~${\DD}$ to the category of~$B$-algebras defines a determinant~${\DD}_B:B[G] \rightarrow B$ on~$B.$ We call~${\DD}_B$ the base change of~${\DD}$ to~$B.$

\subsection{Determinants of degree~\texorpdfstring{$d = 2$}{d=2}}
Given a determinant~${\DD}: A[G] \rightarrow A$ of degree~$2$, the corresponding
characteristic polynomials~$P(\DD,m) \in B[X]$ for~$m \in B[G]$  have degree~$2$ and can be written
in the form
$$P(m) = P(\DD,m) = X^2 - T(m) X + D(m),$$
for maps~$T, D: B[G] \rightarrow B$. Note that the family of maps ~$D:B[G] \rightarrow B$ as~$B$ ranges over all~$A$-algebras is precisely
the data which defines the polynomial law~$\DD$.
In practice, our groups~$G$ will always be Galois groups with the usual pro-finite topology,
and our rings~$A$  will either be ~$p$-adically complete semi-local~$W(k)$-algebras with the~$p$-adic topology
or~$p$-adic fields with the~$p$-adic topology. We insist that all Galois representations
and all determinants considered in this paper are continuous with respect to the topologies on~$G$ and~$A$.

In residue characteristic different from~$2$ and degree~$2$, one can recover~$D$ from~$T$ via the identity
$$D(\sigma) = \frac{T(\sigma)^2 - T(\sigma^2)}{2}.$$
On the other hand, for any~$p$, one can recover~$T$ from~$D$ by the formula
$$T(\sigma) = D(\sigma + 1) - D(\sigma) - 1.$$
We have the following characterization of determinants of degree~$2$.

\begin{lemma}{\cite[Lemma~7.7]{Chenevier}} The set of determinants of ~$G$ over~$A$ of degree~$2$
are in bijection with maps~$(T,D)$ from~$G$ to~$A$ satisfying the following two conditions:
\begin{enumerate}
\item ~$D: G \rightarrow A^{\times}$ is a homomorphism,
\item ~$T: G \rightarrow A$ is a function with~$T(1) = 2$, and such that,
for all~$g$, ~$h \in G$:
\begin{enumerate}
\item ~$T(g h) = T(h g)$,
\item ~$D(g) T(g^{-1} h) - T(g) T(h)  + T(g h) = 0$.
\end{enumerate}
\end{enumerate}
\end{lemma}

In light of this lemma, we shall (from now on) regard a determinant~$\DD$ of~$G$
over~$A$ of degree~$2$ as precisely given by a pair of functions~$(T,D)$
satisfying the equations above. 
Given ~$g \in G$, we have a corresponding characteristic
polynomial~$P(g) = X^2 -T(g) X + D(g)$. By abuse of notation,
we shall denote the pair~$(T,D)$ by~$P = (T,D)$.
By~\cite[Lemma~7.7]{Chenevier}, the functions~$T$ and~$D$ extend
to functions from~$A[G]$ to~$A$. In the case of~$T$, this extension is the linear
extension, and in the case of~$D$, it can be constructed explicitly by using
the equation for~$D(xt + ys)$ given below.
Note that~$D$ as a function of~$A[G]$ determines~$T$ and hence~$P$ and hence~$\DD$,
but~$D$ as a function of~$G$ (in general) does not. Under this equivalence, the base change of a determinant ~$P := (T,D)$ to an~$A$-algebra~$B$ corresponds to the determinant~$f\circ P :=(f\circ T,f \circ D)$ obtained by post-composing the functions~$T$ and~$D$ with the structure homomorphism~$f:A \rightarrow B.$

 If ~$A$ is an algebraically closed field, then
$(T,D)$ may be realized as the trace and (classical) determinant of an actual semisimple 
representation (Theorem~A of~\cite{Chenevier}).

There is a well-defined notion of the kernel of ~$P$ (see~\cite[~\S1.4]{Chenevier}), which in our case has
the following simple description:

\begin{lemma} \label{lemma:kernel} The kernel of a determinant~$P = (T,D)$ of degree~$2$ consist of
the elements~$x \in A[G]$ satisfying the following two conditions:
\begin{enumerate}
\item ~$T(x y) = 0$ for all~$y \in A[G]$,
\item ~$D(x) = 0$.
\end{enumerate}
\end{lemma}

\begin{proof}  For polynomial laws of degree~$2$, we have
(cf.~Example~7.6 of~\cite{Chenevier})
\begin{equation}
\label{eq1}
D(x t + y s) = D(x) t^2 + (T(x) T(y) - T(x y)) ts + D(y) s^2.
\end{equation}
As follows from~\S1.4 of~\cite{Chenevier}, we may
compute the~$x \in \ker(P)$ by finding the~$x$ for which this expression is independent of~$t$.  Taking~$y = 1$ 
yields the equalities~$T(x) = 0$ and~$D(x) = 0$. Returning to the
case of  general~$y$, we then deduce that~$T(x y)
= T(x) T(y) = 0$.
\end{proof}

Suppose that~$H$ is a subgroup of~$G$ such that~$[h] - 1 \in \ker(P)$ for all~$h \in H$.
In this case, by abuse of notation, we say that~$\ker(P)$ contains~$H$. If~$\ker(P)$ contains~$H$, 
then~$[ghg^{-1}] - 1 \in \ker(P)$ for any~$g \in G$, and (cf.~Lemma~7.14 of~\cite{Chenevier})
the determinant~$P$ factors through~$A[G/N]$, where~$N$ is the normal closure
of~$H$.  (That is, the functions~$T$ and~$D$ on~$A[G]$ depend only on their image
in the quotient~$A[G/N]$.)
In particular, to show that a determinant on~$\OL[G_{\Q}]$ is unramified at a prime~$l$
(for example~$l = p$), it suffices to show that the kernel contains some (any) choice of inertia subgroup~$I_l$ at~$l$, or equivalently:

\begin{lemma} \label{lemma:kerneltwo} ~$I_l = H \subset G = G_{\Q}$ lies in the kernel of~$P$ if and only if:
\begin{enumerate}
\item ~$T(hg) = T(g)$ for all~$h \in H = I_l$ and~$g \in G = G_{\Q}$.
\item ~$D(h - 1) = 0$ for all~$h \in H = I_l$.
\end{enumerate}
\end{lemma}

\subsection{Ordinary Determinants} \label{section:ordpsu}

Let~$\OL$ be the ring of integers of a finite extension~$[K:\Q_p] < \infty$, let
$\varpi$ be a uniformizer of~$\OL$, and suppose that~$\OL/\varpi = k$. Let
$\Pbar = (\Tbar,\Dbar): G_{\Q} \rightarrow k$
be a degree~$2$ determinant which is unramified outside~$Np$. In practice,
it will always be taken to be modular of level~$\Gamma_1(N)$.
Let us fix, once and for all, an embedding of~$\Qbar$ into~$\Qbar_p$, and hence inclusions:
$$I_p \subset D_p \subset G_{\Q},$$
where~$I_p$ is the inertia group of~$\Q_p$,  and~$D_p = \Gal(\Qbar_p/\Q_p)$ is the decomposition group.
Let us also fix a Frobenius element~$\phi \in D_p$. There is a natural projection~$D_p \rightarrow D_p/I_p \simeq \widehat{\Z}$ whose image is
topologically generated by the image of~$\phi$.
Let~$\eps: G_{\Q} \rightarrow \Z^{\times}_p$ be the cyclotomic character; we may choose~$\phi$ so that~$\eps(\phi) = 1$.
Enlarging~$k$ if necessary, let~$\alphabar$ and~$\betabar$ be the roots of the quadratic polynomial
$$X^2 - \Tbar(\phi) X + \Dbar(\phi) = 0$$
over~$k$. 
We do not assume that these are necessarily distinct.

There are a number of slightly different definitions of ordinary Galois
representations in the literature.  Let us say that
a~$2$-dimensional representation~$\rho: G_{\Q_p} \rightarrow \GL_2(\overline{\Q}_p)$ is  ordinary  if the underlying~$2$-dimensional vector space~$V$ admits a two step filtration~$0 \subsetneq V'  \subsetneq V$ such that the action
of~$G_{\Q_p}$ on~$V'' = V/V'$ is unramified. (This coincides, for example,
with the definition of ordinary in~\cite{SkinnerWiles}.)
We furthermore say that~$\rho$ is ordinary of weight~$n$ if the action of~$G_{\Q_p}$
on~$V'$ is via an unramified twist of~$\eps^{n-1}$. By abuse of notation,
{if~$\rho: G_{\Q} \rightarrow \GL_2(\Qbar_p)$} is a global Galois representation, we say that it is ordinary
if~$\rho |_{G_{\Q_p}}$ is ordinary (respectively, ordinary of weight~$n$).
When a representation is ordinary, various relations are imposed on its associated determinant. We collect several of these relations common to all ordinary~$2$-dimensional representations of weight~$n$,
and then define that a determinant~${P = (T,D): A[G_{\Q}] \rightarrow A}$ of degree~$2$ 
to be an ``ordinary determinant of weight~$n$'' if and only if it satisfies
these conditions. Our definition includes the auxiliary data of an ``eigenvalue'' ~$\alpha \in A^{\times}$ of the Frobenius element~$\phi.$ This ``eigenvalue'' satisfies some relations shared by every value which occurs as the eigenvalue of~$\phi$ on a choice of unramified quotient of~$\rho|_{G_{\Q_p}}$ in an 2-dimensional ordinary representation of weight~$n$. We will be interested in deformations of~$\Pbar$ to Artinian local rings~$(A,\m)$ which are ordinary of weight~$n$.

 \begin{df} \label{ord} Let ~$(A,\m)$ be a Noetherian local ring with residue field ~$k$.
An  ordinary determinant ~$P: A[G_{\Q}] \rightarrow A$ of degree~$2$ and weight ~$n$ with eigenvalue ~$\alpha \in A^{\times}$ 
consists of a  pair~$(P,\alpha)$ 
where~$P = (T,D): A[G_{\Q}] \rightarrow A$ is a degree~$d = 2$
 determinant
satisfying the following properties:
\begin{enumerate}
\item ~$P(h) = (X - 1)(X - \psi(h))$ 
 for all ~$h \in I_p$, 
where ~$\psi = \eps^{n-1}$.
\item   ~$\alpha$ is a root of ~$X^2 - T(\phi) X + D(\phi)$.
 \item For all~$h \in I_p$,  $(h - \psi(h))(\phi - \alpha) \in \ker(P)$. 
Equivalently,  for all~$g \in G_{\Q}$ and~$h \in I_p$, 
 ~$$T(g(h - \psi(h))(\phi - \alpha)) =  T(g h \phi) - \psi(h) T(g \phi) - T(g h) \alpha + T(g) \psi(h) \alpha = 0.$$ 
 \end{enumerate}
\end{df}

The first two conditions of this definition are self-explanatory. The last may
be  somewhat surprising to the reader; note that
it involves a condition on general elements~$g \in G_{\Q}$ rather than simply being a condition on the decomposition group.
This turns out to be necessary, because the determinant (or pseudo-character)  associated to the decomposition group of a locally reducible representation does not
know which character comes from the quotient and which comes from the submodule. 
The idea behind this definition, as we shall see shortly below,
is to capture the notion that the product~$(h - \psi(h))(\phi - \alpha)$ is \emph{identically} zero, rather than just of the form~$\left( \begin{matrix} 0 & * \\ 0 & 0 \end{matrix}\right)$.
There is presumably a close relationship between this definition and the definition of
ordinary pseudo-characters in Wake,~Wang-Erickson  (see~\cite{wake2015ordinary}
and~\S7.3 of~\cite{WE}), although in our context it
is important that we can work in non-$p$ distinguished situations by choosing an eigenvalue of Frobenius, which amounts to a partial resolution
of the corresponding deformation rings (presumably such modifications could also be adapted to~\cite{WE}). On the other hand, we do exploit the crucial idea due to Wang-Erickson
that the notion of ordinarity for pseudo-representations should be a global rather than local
condition.
The following lemma provides a justification for the final condition above, and the proof
provides a motivation for its definition.

\begin{lemma}\label{lemma:classical} Suppose that ~$f$ is a classical  modular eigenform 
 of level ~$\Gamma_0(p) \cap \Gamma_1(N)$ with Nebentypus character~$\chi$ of
weight ~$n \ge 2$ with coefficients in ~$\OL$, and suppose that ~$\alpha$ is the ~$U_p$-eigenvalue of ~$f$.
Assume that~$f$ is ordinary (equivalently, that~$\alpha$ has trivial valuation). Then the associated
determinant~$P_f: \OL[G_{\Q}] \rightarrow \OL$
is ordinary with eigenvalue ~$\alpha$, weight~$n$, and is unramified outside~$Np$.
\end{lemma}

Note that~$f$ in Lemma~\ref{lemma:classical} need not be new at either the prime~$p$ or primes dividing~$N$.

\begin{proof} Since ~$\OL$ has characteristic zero,  there is a Galois representation (via \cite{Deligne}) 
$$\rho_f: G_{{\Q,Np}} \rightarrow \GL_2(\overline{\Q}_p)$$
associated to~$f.$ The determinant~$P_f = (T_f,D_f)$ is (by definition) the determinant associated to the representation~$\rho_f.$    
Since~$\rho_f$ factors through~$G_{\Q,Np}$, this determinant is unramified at primes outside~$Np$.
Let~{$\lambda_{\alpha}: G_{\Q_p} \rightarrow \overline{\Q}_p^{\times}$}
denote the unramified character which sends~$\Frob_p$ to~$\alpha$.
We collect the following facts concerning the Galois representation~$\rho_f$:
\begin{facts} 
\label{facts} The representation~$\rho_f$ has the following properties:\footnote{\textit{Some References:}
The fact that~$\rho_f$ is unramified outside~$Np$ already follows
from the original construction of Deligne~\cite{Deligne}. Since the Nebentypus character has
conductor dividing~$N$, the corresponding Galois representation~$\chi$ is
certainly unramified outside~$N$. The
second claim follows immediately from~\cite[Theorem~2]{WilesOrdinary}.
Consider the third claim, so we are assuming that~$f$ is new at~$p$.
If one writes~$\chi = \chi_p \chi_{N}$ where~$\chi_p$ and~$\chi_N$ are characters
corresponding to the identification~$(\Z/Np \Z)^{\times} = (\Z/p \Z)^{\times} \oplus (\Z/N \Z)^{\times}$,
then (by assumption)~$\chi_p$ is trivial.
It follows (see~\S1 of~\cite{AL}) that~$f$ is an eigenform for operator~$W_p$
with eigenvalue~$\lambda_p(f)$ satisfying~$\lambda^2_p(f) = \chi(p)$ (\cite[Proposition ~1.1]{AL}). On the other hand,  
by~\cite[Theorem~2.1]{AL}, we deduce  that~$\alpha^2 =  \lambda^2_p(f) p^{n - 2} = \chi(p) p^{n-2}$.
Under our assumption that~$\alpha$  is a~$p$-adic unit, this can only occur when the weight~$n = 2$.
When~$n = 2$, however, we can appeal to~\cite[Theorem~3.1(e)]{DDT} which gives
a detailed description of the local properties of Galois representations associated to
ordinary forms. Finally, the identification of~$\chi |_{D_p}$ with~$\lambda^2_{\alpha}$
follows either by considering determinants or the identity~$\alpha^2 = \chi(p)$ discussed above.} 
\begin{enumerate}
\item The representation~$\rho_f$ is unramified outside~$Np$. The
trace and (classical) determinant of~$\rho_f(\Frob_{\ell})$ are equal to~$a_{\ell}(f)$ and~$\ell^{n-1} \chi(\ell)$
respectively. The (classical) determinant of~$\rho_f$ is the character~$\chi \eps^{n-1}$, where~$\chi$
is unramified outside~$N$.
\item If~$f$ is old at level~$p$, and the corresponding eigenform~$g$ of level~$\Gamma_1(N)$ has~$T_p$ eigenvalue~$a_p$,
then~$\alpha$ is the unit root of~$X^2 - a_p X + p^{n-1} \chi(p)$, and
$$\rho_f |_{D_p}  = \rho_g |_{D_p} \sim \left( \begin{matrix} \eps^{n-1} \lambda^{-1}_{\alpha} \chi & * \\
0 & \lambda_{\alpha} \end{matrix} \right).$$
\item \label{facts:new} If~$f$ is new at level~$p$, then~$n = 2$, 
$$\rho_f |_{D_p} \sim \left( \begin{matrix} \eps \lambda_{\alpha} & * \\
0 & \lambda_{\alpha} \end{matrix} \right),$$
and~$\chi |_{D_p} \simeq \lambda^2_{\alpha}$.
\end{enumerate}
\end{facts}

Using these properties, we see that the required conditions for~$P_f$ to be ordinary
with eigenvalue~$\alpha$ are easily met with the possible exception of the 
final condition. For this, note  that from the explicit descriptions above there
exists a basis such that:
$$\rho_f |_{I_p} = \left(\begin{matrix} \psi & * \\ 0 & 1 \end{matrix}\right), \qquad
\rho_f(\phi) = \left(\begin{matrix} \chi(\phi) \alpha^{-1} &   *  \\ 0 & \alpha \end{matrix} \right),$$
where~$\det(\rho_f) = \eps^{n-1} \chi$.
We find, 
with~$h \in I_p$, that,  in ~$M_2(\OL)$,
$$(\rho_f(h) - \psi(h))(\rho_f(\phi) - \alpha) = \left( \begin{matrix} 0 & 0 \\ 0 & 0 \end{matrix} \right).$$
It follows that
 ~$$ T_f(s(h - \psi(h))(\phi - \alpha)) =   \tr(\rho_f(s)(\rho_f(h) - \psi(h))(\rho_f(\phi) - \alpha)) = 0$$ for all~$s\in G_{\Q}.$\end{proof}

We now fix our choice of~$\Pbar$. Let ~$\Pbar: k[G_{\Q}] \rightarrow k$ be the determinant associated to
a mod~$\varpi$ weight one eigenform~$g$ of level~$\Gamma_1(N),$ i.e.~the determinant associated to the Galois representation classically attached to $g$ \cite[Proposition  11.1]{gross1990tameness}. Suppose that~$g$ has Nebentypus character~$\chi$ and~$T_p$-eigenvalue~$a_p$,  and  let~$\alphabar$ and~$\betabar$ be the roots of~$X^2 - a_p X + \chi(p),$ which we assume (enlarging~$\mathcal{O}$ if necessary) are~$k$-rational.

\begin{lemma} \label{implicit} Let~$n \equiv 1 \mod (p-1)$ be an integer. The determinant~$\Pbar$ is ordinary of degree~$2$
and weight~$n$ with eigenvalue~$\alphabar$ and is unramified outside~$Np$.
If~$n > 1$, there is an eigenform~$f$ of level~$\Gamma_1(N) \cap \Gamma_0(p)$ 
and weight~$n$ which is ordinary at~$p$ for which the~$U_p$-eigenvalue of~$f$ is congruent to~$\alphabar$ mod~$\varpi$ such that~$\Pbar = \Pbar_f.$
\end{lemma}

By symmetry, the result holds with~$\alphabar$ replaced by~$\betabar$.

\begin{proof} Since~$\eps^{n-1}$ is trivial mod~$\varpi$ for~$n \equiv 1 \mod (p-1)$, if~$\Pbar$
is ordinary with eigenvalue~$\alphabar$ for one such~$n$, it is ordinary for all such~$n$.
Suppose we can construct an eigenform~$h$ modulo~$\varpi$  of 
level~$\Gamma_1(N)$
of weight~$p$ and such that the~$T_p$-eigenvalue of~$h$ is congruent to~$\alphabar$ mod~$\varpi$,
and such that~$\Pbar_h = \Pbar$. By multiplying by powers of the Hasse invariant, we 
deduce that there also exists such a form in any weight~$n \equiv 1 \mod (p-1)$ such that~$n > 1$.
All mod~$\varpi$ modular forms in weights~$n > 1$ and level~$\Gamma_1(N)$  lift to characteristic zero. (This follows as in~\cite[Theorem~1.7.1]{Katz}, 
the running assumption that~$N \ge 5$ guaranteeing that~$X_1(N)$ is a fine moduli space.)
Moreover, using the Deligne--Serre lifting lemma (\cite[Lemme~6.11]{DS}), one can always choose a lift~$h$ which is an eigenform for all the Hecke operators.
The lifted form~$h$ of weight~$\Gamma_1(N)$ has weight~$n > 1$ and~$T_p$-eigenvalue~$\alphabar \mod \varpi$. But now the ordinary stabilization~$f$ of~$h$ of level~$\Gamma_1(N) \cap \Gamma_0(p)$ has has~$U_p$-eigenvalue~$\alphabar \mod \varpi$, and~$\Pbar_f = \Pbar_h = \Pbar$, as required. Finally, we deduce from Lemma~\ref{lemma:classical}  applied to~$f$ that~$\Pbar$ is ordinary with
eigenvalue~$\alphabar$ (of weight~$n$ and unramified outside~$Np$).
Thus it remains to construct~$h$ from~$g$.

If~$A$ is the Hasse invariant, then~$Ag$ is a modular form
mod~$\varpi$ of level~$\Gamma_1(N)$ and weight~$p$ which is an eigenform
for all Hecke operators except for~$T_p$, and moreover has the same eigenvalues as~$g$. The same is true of~$T_p(Ag)$ and also~$A(T_p g)$
(the latter is just~$a_pAg$). On the level of~$q$-expansions, there are equalities
$Ag = g$ and~$A T_p(g) - T_p(A g)  = \Vg$ respectively. 
Hence~$h = g - \betabar \Vg$ is a weight~$p$
 modular eigenform mod~$\varpi$ of level~$\Gamma_1(N)$ with~$\Pbar_f = \Pbar$
 and with~$T_p$-eigenvalue~$\alphabar$. 
  To see that~$T_p h = \alphabar h$, note (cf.~\cite[\S4]{gross1990tameness},
 especially~(4.7))
  that~$T_p(\Vg) = Ag$
 and~$T_p Ag = a_p Ag -\chi(p) \Vg$, and hence
  $$\begin{aligned} T_p h = & \  T_p (A g - \betabar \Vg)  \\
  = & \  a_p A g - \chi(p) \Vg - \betabar A g  \\
= & \  (\alphabar + \betabar) Ag - \alphabar \betabar \Vg - \betabar Ag \\
= & \  \alphabar(A g - \betabar \Vg) = \alphabar h. \end{aligned}
$$
\end{proof}

Let~$R = R^{\univ}$ denote the  universal deformation ring of~$\Pbar$ (cf.~\cite[Proposition ~7.59]{Chenevier}) unramified outside~$Np$. It pro-represents
the functor which, for Artinian local~$W(k)$-algebras~$(A,\m)$ with residue field~$A/\m = k$,
consists of determinants~$P = (T,D)$ valued in~$A$ whose mod~$\m$ reduction is~$\Pbar$.
Let~$P^{\univ} = (T^{\univ},D^{\univ})$ denote the corresponding universal determinant.
 We define a mild variant on this ring by 
 considering such determinants: ~$P = (T,D): A[G_{\Q}] \rightarrow A$ 
 \emph{together} with a root ~$\alpha$ of ~$X^2 - T(\phi) X + D(\phi)$. The result is an extension ~$\Rtw$ of ~$R$ given by
$$\Rtw = R[\alpha]/(\alpha^2 - T^{\univ}(\phi) \alpha + D^{\univ}(\phi)).$$
The ring~$R$ is a local~$W(k)$-algebra, but the ring~$\Rtw$ is a semi-local~$W(k)$-algebra with either one or two maximal ideals. It has~$2$
maximal ideals precisely when the polynomial~$\alpha^2 - \overline{T}(\phi) \alpha + \overline{D}(\phi) \in k[\alpha]$ is
separable. 

\begin{df} \label{semi} Let~$\Ddag_n(A)$ denote the functor which, for Artinian local rings~$(A,\m)$ with residue field~$A/\m = k$,
consists of ordinary determinants~$(P,\alpha_0)$ of weight~$n$ unramified outside~$Np,$ where~$P$ is a deformation of~$\Pbar$ to $A$, and~$n \equiv 1 \mod (p-1)$ is a positive integer.
\end{df}

Note that elements in~$\Ddag_n(k)$ are in bijection with choices of ~$\alphabar \in k$ so that~$\Pbar$ is ordinary of weight~$n$ with eigenvalue~$\alphabar$.
By Lemma \ref{implicit}, such a choice of eigenvalue exists. Furthermore since~$\alphabar$ is a root~$X^2 - \Tbar(\phi) X+ \Dbar(\phi),$ the size of ~$\Ddag_n(k)$ is at most 2. For each root ~$\alphabar \in k$ of ~$X^2 - \Tbar(\phi) X+ \Dbar(\phi),$ consider the sub-functor ~$\widetilde{D}_n^{\dagger,\alphabar}(A) \subseteq \Ddag_{n}(A)$ consisting of pairs with ~$(P,\alpha_0)$ such that~$\alpha_0 \equiv \alphabar \mod \m.$ The functor~$\Ddag_{n}$ decomposes as the coproduct 
~$$\Ddag_{n}(A) = \coprod_{(\overline{P},\alphabar) \in \Ddag_n(k)} \widetilde{D}_n^{\dagger,\alphabar}(A),$$ and each of the sub-functors~$\widetilde{D}_n^{\dagger,\alphabar}$ are pro-represented by a (potentially trivial) Noetherian local~$W(k)$-algebra~$\Rtw_n^{\dagger,\alphabar}.$  By abuse of terminology, we will say~$\Ddag_n$ is pro-represented by the semi-local ring~$$\Rtd_n := \bigoplus_{(\overline{P},\alphabar) \in \Ddag_n(k)}\Rtw_n^{\dagger,\alphabar}.$$ Explicitly, if~$\wP^{\univ}$ is the base change of~$P^{\univ}$ to the~$R$-algebra~$\Rtw,$ then~$\Rtd_n$ is the quotient of~$\Rtw$ by the ideal generated by all the relations which obstruct~$P^{\univ}$ from being ordinary of weight~$n$ with eigenvalue~$\alpha.$ The universal determinant~$P_n^{\dag,\univ}$ is base change of~$\wP^{\univ}$ to~$\Rtd_n$ and the universal eigenvalue is~$\alpha.$

The determinant~$P^{\dag,\univ}$ itself is valued in the subring~$\Rd_n$ of~$\Rtd_n,$ which is the image of~${R \subset \Rtw}$. However,
the element~$\alpha$ will not, in general, lie in~$\Rd_n$. The extra data of~$\alpha$ records, implicitly, the ``choice'' of realizing the corresponding determinant as ordinary. (The same determinant~$P$
can in principle be realized as an ordinary determinant~$(P,\alpha)$ for different values of~$\alpha$.)

The following result is the key proposition which
allows us to prove that certain ordinary determinants are unramified. The idea is that, given a representation
which is ordinary, the more the representation is ramified, the more the choice of ordinary eigenvalue~$\alpha$ is pinned down
by the Galois representation, because the ramification structure gives a partial filtration on the representation which mirrors
the ordinary filtration. The extreme case, in which~$\alpha$ cannot be distinguished from the other root~$\alpha^{-1} D(\phi)$
of the characteristic  polynomial of~$\phi$, should only occur when the representation is unramified. While these claims
are obvious for~$\Qbar_p$-valued representations, the key property of our definition is that one can prove this for any quotient
of~$\Rtd_n$.

\begin{prop} \label{prop:key}
Let~$\Rtd_n \rightarrow \Sw$ be a surjective homomorphism of~$W(k)$-algebras,  
and let~$S$ denote the image of~$\Rd_n$ in~$\Sw$. Suppose that~$\Sw/S$ is a free~$S$-module
of rank one, or equivalently, that the annihilator of~$\Sw/S$ as an~$S$-module is trivial.
Then the corresponding
determinant~$P$ valued in~$S$ is unramified.
\end{prop}

\begin{proof} 
We first verify that~$D(h - 1) = 0$ without any assumptions.
From first condition of Definition~\ref{ord} 
we see that~$D(h) = \psi(h)$ and $T(h) = 1 + \psi(h)$, and thus,
from Equation~\eqref{eq1} in the proof
of Lemma~\ref{lemma:kernel}, we deduce that
$$D(h - 1) = D(h) - (T(h)T(1) - T(h)) + D(1) = \psi(h) - (\psi(h) + 1) + 1 = 0.$$

We now turn to the second condition of
Lemma~\ref{lemma:kerneltwo}.
The module~$\Sw/S$ is a cyclic~$S$-module generated by~$\alpha$, so it is free
if and only if the annihilator of~$\alpha$ is trivial. 
We have by definition the identity (for~$s \in G_{\Q}$ and~$h \in I_p$ and~$\psi = \eps^{n-1}$)
 ~$$T(s h \phi) - \psi(h) T(s \phi) - T(s h) \alpha + T(s) \psi(h) \alpha = 0.$$
 We may re-arrange this to obtain the identity:
 ~$$\alpha (T(sh) - T(s) \psi(h)) = T(s h \phi) - \psi(h) T(s \phi).$$
  Note that the value~$T(s)$
for any~$s \in G_{\Q}$ lands in~$S$, as does the image of any element of~$W(k)$,
and hence it follows that
$$ \alpha (T(sh) - T(s) \psi(h)) = 0 \in \Sw/S.$$
 Take~$g$ to be the identity, so~$T(s h) = T(h) = 1 + \psi(h)$ and~$T(s) = 2$. Then we deduce that
 ~$$ \alpha (1 - \psi(h)) = 0 \in \Sw/S$$
 for all~$h \in I_p$.
 If~$\Sw/S$ is free, then its annihilator of~$\alpha$ is trivial, and thus~$\psi(h) = 1$ for all~$h$. But we 
 then deduce 
 for the same reason that
 ~$T(sh) - T(s) \psi(h) = T(sh) - T(s) = 0$
 for all~$s \in G_{\Q}$ and~$h \in I_p$, from which it follows by Lemma~\ref{lemma:kerneltwo} 
 (note that~$T(s h) = T(h s)$) that~$I_p$
 is contained in the kernel.
\end{proof}

\section{Galois Deformations}

By Lemma \ref{implicit}, our fixed determinant  ~$\Pbar = (\Tbar,\Dbar): k[G_{\Q}] \rightarrow k$  is associated to 
an ordinary mod~$\varpi$
eigenform of level~$\Gamma_1(N) \cap \Gamma_0(p)$ in each weight~$n\geq 2$ satisfying~$n \equiv 1 \mod p-1.$ 
 Given our choice of Frobenius element~$\phi \in D_p \subset G_{\Q}$, recall that ~$\alphabar$ and~$\betabar$ are the roots of the polynomial
$$\Pbar(\phi) = X^2 - \Tbar(\phi) X + \Dbar(\phi).$$
 We start by considering determinants arising from forms of higher weight. 

\begin{lemma} \label{lemma:higher}
Let~$n \geq 2$ be an integer such that ~$n \equiv 1 \mod p-1.$ 
Let~$\wT_n$ denote the~$\OL$-algebra of endomorphisms
of
$$M_n(\Gamma_0(p) \cap \Gamma_1(N),\OL)$$
 generated
by Hecke operators~$T_{\ell}$  and~$\langle \ell \rangle$ for~$\ell$ prime to~$Np$, 
together with~$U_p$. 
 Let ~$\m$ denote the ideal of~$\wT_n$ generated by~$\varpi$ and by any lift in~$\wT_n$
 of the following elements of~$\wT_n/\varpi$: the operators ~$T_{\ell} - \Tbar(\Frob_{\ell})$ 
 and~$\langle \ell \rangle \ell^{n-1} - \Dbar(\Frob_{\ell})$  for~$(\ell,Np) = 1$,
and~$(U_p - \alphabar)(U_p - \betabar)$.  Assume that~$\Pbar$ is associated to
an ordinary mod~$\varpi$
eigenform of level~$\Gamma_1(N) \cap \Gamma_0(p)$ and weight~$n$ with~$U_p$-eigenvalue
congruent to either~$\alphabar$ or~$\betabar$ modulo~$\varpi$, 
 so that~$\m$ is a proper ideal. Then there exists a canonical surjection
of semi-local rings
$$\Rtd_n \rightarrow \wT_{n,{\m}}$$
sending~$\alpha \in \Rtd_n$ to~$U_p$.
\end{lemma}

\begin{remark} \label{semilocal} \emph{
If~$\T_n \subset \wT_n$ denotes the subring generated by the  all the Hecke operators except~$U_p$, then~$\m \cap \T_n$ is maximal. However,~$\m$ itself need not be maximal.
Throughout the rest of the paper, we let~$\wT_{n,{\m}}$ denote the
completion~$\wT_{n,{\m}}:= \projlim{\wT_{n}/\m^r}$ --- it need not be a local ring.
The Hecke algebra~$\wT_{n,{\m}}$ is non-local precisely when~$\alphabar \ne \betabar$ and when~$\Pbar$ is
associated to an ordinary mod~$\varpi$ eigenform with~$U_p$-eigenvalue congruent to~$\alphabar \mod \varpi$
and is \emph{also} associated to an eigenform with~$U_p$-eigenvalue congruent to~$\betabar \mod \varpi$. In that case, the ideals~$\m_{\alphabar}$ and~$\m_{\betabar}$ obtained
by adjoining any lift of~$U_p - \alphabar$ or~$U_p - \betabar$  respectively from~$\wT_n/\varpi$ to~$\m$ are both maximal, 
and there is an isomorphism~{$\wT_{n,\m} \cong \wT_{n,\m_{\alphabar}} \oplus  \wT_{n,\m_{\betabar}}$}. Working with semi-local rings
allows us to treat the cases~$\alphabar = \betabar$ and~$\alphabar \ne \betabar$ simultaneously.
If~$M$ is a module for~$\wT_n$, then, when~$\m$ is not maximal, there is also a corresponding
 identification~$M_{\m} := \projlim M/\m^r =  M_{\m_{\alphabar}} \oplus  M_{\m_{\betabar}}$.
}
\end{remark}

\begin{proof}[Proof of Lemma~\ref{lemma:higher}]  Consider an embedding~$K \rightarrow L$, where~$L$ is a field
which contains the eigenvalues of all elements of~$\wT_n$.
The Hecke algebra~$\wT_n$ 
 acts faithfully on ~{$M_n(\Gamma_0(p) \cap \Gamma_1(N),L)$}.
Recall that~$\T_n \subset \wT_n$ denotes the subring generated by Hecke operators away from~$Np$
(i.e.~without~$U_p$).
For each newform~$h$ which contributes to~$M_n(\Gamma_0(p) \cap \Gamma_1(N),L)$,
there is a corresponding vector space~$V(h) \subset M_n(\Gamma_0(p) \cap \Gamma_1(N),L)$ generated by~$h$ together with the oldforms associated to~$h$. (The space~$V(h)$
can also be identified with the invariants~$\pi^{\Gamma_1(N) \cap \Gamma_0(p)}$,
where~$\pi$ is the smooth admissible~$\GL_2(\mathbf{A}^{(\infty)})$-representation over~$L$ generated by~$h$.)
 There is a~$\T_n$-equivariant
isomorphism
$$M_n(\Gamma_0(p) \cap \Gamma_1(N),L) \simeq \bigoplus_{g} V(h),$$
where~$\T_n$ acts on~$V(h)$ through scalars corresponding to the
homomorphism~$\eta_h: \T_n \rightarrow L$ sending~$T_{\ell}$ to~$a_{\ell}(h)$
and~$\langle \ell \rangle$ to~$\ell^{n-1} \chi(\ell)$ where~$\chi$ is the Nebentypus  character of~$h$.
Let us now consider the action of the operator~$U_p$. For each map~$\eta_h: \T_n \rightarrow L$
(which corresponds to a fixed Galois representation~$\rho_h$) one of the following
two things happens:
\begin{enumerate}
\item The newform~$h$ has level~$\Gamma_0(p)$ at~$p$, in which case~$U_p$ acts on~$V(h)$ via a scalar.
\item The newform~$h$ has level~$\Gamma_0(1)$ at~$p$, in which case~$U_p$ acts on~$V(h)$
and satisfies the identity~$U^2_p - a_p U_p + p^{n-1} \chi(p) = 0$.
\end{enumerate}
In particular, the algebra~$\wT_n$ will always acts semi-simply in the first
case and act semi-simply in the second case as long
as the corresponding polynomial~$X^2 - a_p X + p^{n-1} \chi(p)$ has distinct roots.
This is known in general only under the assumption of the Tate conjecture
 (cf.~\cite{Coleman}), but it can certainly only fail to happen
when~$a^2_p = 4 p^{n-1} \chi(p)$, which would force the (multiple) eigenvalue of~$U_p$
to have positive valuation (since~$n \ge 2$). In particular, such forms do not contribute to
$M_n(\Gamma_0(p) \cap \Gamma_1(N),\OL)_{\m} \otimes_{\OL} L$, because
(since~$\m$ contains the preimage of~$(U_p - \alphabar)(U_p -\betabar)$
for non-zero~$\alphabar$ and~$\betabar$) the element~$U_p$
acts invertibly on this space. (Recall, following Remark~\ref{semilocal},
that when~$\m$ is contained in two primes, $ M_n(\Gamma_0(p) \cap \Gamma_1(N),\OL)_{\m}$ is simply
the direct sum of~$M_n(\Gamma_0(p) \cap \Gamma_1(N),\OL)_{\m_{\alphabar}}$ and~$M_n(\Gamma_0(p) \cap \Gamma_1(N),\OL)_{\m_{\betabar}}$.)
It follows
that there is an injection
$$i_n: \wT_{n,\m} \hookrightarrow \bigoplus_{f} L,$$
where the sum ranges 
over all~$\wT_n$-eigenforms ~$f \in M_n(\Gamma_0(p) \cap \Gamma_1(N),L)$ such that~$\Pbar_f = \Pbar$ and the~$U_p$-eigenvalue is congruent either to~$\alphabar$ or~$\betabar$.
We identify~$\wT_{n,\m}$ with its image under~$i_n.$ For each of the forms~$f$ above, denote the~$U_p$-eigenvalue by~$\alpha(f)$. By Lemma~\ref{lemma:classical}, the determinants~$P_f$  are ordinary with
eigenvalues~$\alpha(f)$, weight~$n$, and unramified outside~$Np$. Hence, for each form~$f$ there is a homomorphism
$$i_f:\Rtd_n \rightarrow L$$ such that~$i_f\circ P^{\dag,\univ} = P_f$ and which maps~$\alpha$ to~$\alpha(f).$  Taking the direct sum of the maps~$i_f,$ we obtain a homomorphism~$$j_n: \Rtd_n \rightarrow \bigoplus_{f} L$$ under which~$\alpha$ maps to~$U_p$,
$T(\Frob_{\ell})$ maps to~$T_{\ell},$ and~$D(\Frob_{\ell})$ maps to~$\langle \ell \rangle.$ We conclude that~$j_n$ factors through a surjective homomorphism ~$$\Rtd_n \rightarrow \wT_{n,\m}$$ under which ~$\alpha$ maps to~$U_p.$
\end{proof}

We are now ready to prove the main theorem.

\begin{proof}[Proof of Theorem~\ref{theorem:main}] Recall that~$\T = \T_1$ is the $\mathcal{O}$-subalgebra of ~$\mathrm{End}
_{\mathcal{O}}H^0(X_1(N),\omega \otimes K/\mathcal{O}
)$ generated by~$T_{\ell}$ and~$\langle \ell \rangle$ for~$(\ell,N) = 1$. This ring contains~$T_p$, but the element~$T_p$ in weight one is also generated by the other Hecke operators
(see, for example, Lemma~3.1 of~\cite{calegari2015non}). For each positive integer $m,$ let ~$\T(m)$ denote the image $\T$ in ~$\End_{\mathcal{O}} H^0(X_1(N),\omega/\varpi^m).$ 
The ring ~$\T \cong \varprojlim \T(m).$ Therefore, to prove
 Theorem~\ref{theorem:main}, it suffices to construct for each $m > 0$ a degree~$d = 2$ 
determinant $$\DD_m: \T(m)[G_{\Q}] \rightarrow \T(m), \qquad P(\DD_m,\sigma) = X^2 - T_m(\sigma) X + D_m(\sigma),$$
which is unramified outside ~$N\infty,$ and such that for all primes ~$\ell \nmid N$ (including $\ell =p)$ the characteristic polynomial of $\Frob_{\ell}$ satisfies  
 ~$$T_m(\Frob_{\ell}) = T_{\ell} \ \text{and} \ D_m(\Frob_{\ell}) = \langle \ell \rangle.$$ 
 In the remainder of the proof, we will assume that $m>0$ is fixed, and will denote by an abuse of notation ~$\T(m)$ by~$\T.$ 

There is a decomposition~$\T = \bigoplus \T_{\m}$
over the maximal ideals~$\m$ of~$\T.$ Hence, it suffices to construct the desired determinant after completing at a maximal ideal~$\m$ of~$\T$. Let~$\Pbar$ denote our fixed modular residual determinant, which we have assumed is supported in weight one, and let~$\m$ denote the maximal
ideal which is the kernel of the corresponding map~$\T \rightarrow k$.
Let~$\wT_{n}$ denote the Hecke algebra of Lemma~\ref{lemma:higher} in weight~$n := 1 + p^{m-1}(p-1)$ which contains~$U_p$
(and has coefficients in~$\OL$). By abuse of notation,
we also let~$\m$ denote the ideal of~$\wT_{n}$ defined in Lemma ~\ref{lemma:higher}.  By Lemma~\ref{implicit}, this ideal is proper.

By Lemma~3.16 of~\cite{CG},   there is a surjective map
\begin{equation}\label{doubling} \Rtd_n \twoheadrightarrow \wT_{n,\m} \twoheadrightarrow \Sw:= \T_{\m}[U_p]/(U_p^2 - T_p U_p + \langle p \rangle) \end{equation} (which sends~$T_{\ell}$ and~$\langle \ell \rangle$ to~$T_{\ell}$ and~$\langle \ell \rangle$ respectively, and sends~$U_p$ to~$U_p$, where~$U_p$ in $\Sw$ is viewed
as a formal variable satisfying the given quadratic relation). Although the running assumption
in~\S3 of~\cite{CG} is that~$p > 2$, the proof of~\cite[Lemma~3.16]{CG} applies
(as written with no changes necessary)
 with~$p = 2$. 
The image~$S$ of~$\Rd_n \subset \Rtd_n$ is generated by the values of~$T$ and~$D$ on Frobenius elements, which land inside the ring~$\T_{\m}$ (in fact, they
generate the ring~$\T_{\m}$). But~$\Sw$ is free of rank two over~$\T_{\m}$, and thus~$\Sw/S$ has no annihilator.
Consequently, the corresponding determinant in~$\T_{\m}$ is unramified by Proposition~\ref{prop:key}.
To show that~$T(\Frob_p) = T_p$ and~$D(\Frob_p) = \langle p \rangle$,  it suffices to show that~$T(\phi) = T_p$ and~$D(\phi) = \langle p \rangle$. 
The image of~$\alpha$ in ~$\T_{\m}[U_p]/(U_p^2 - T_p U_p + \langle p \rangle)$ was~$U_p$, which satisfies the equation
$X^2 - T_p X + \langle p \rangle = 0$. Yet~$\alpha$ also satisfies the equation~$X^2 - T(\phi) X + D(\phi) = 0$.
Since this algebra is free of rank two over~$\T_{\m}$, these quadratics must be the same, and hence~$T(\phi) = T_p$ and~$D(\phi) = \langle p \rangle$.
\end{proof}

\begin{remark} \label{rem:original}
\emph{The proof above relies on~\cite[Lemma~3.16]{CG}.  
We also note, however, that the content of this lemma is simply
an alternate form of doubling
which is a already implicit in the work of Wiese~\cite{Wiese}. }
\end{remark}

\begin{remark} \emph{One should also be able to apply the methods
of this paper in the case~$\qq \ne p$ when~$\qq$ exactly divides~$N$,
where now one wants to capture in this context the notion of a determinant
  ``admitting an unramified quotient line''  when restricted to the inertia group~$I_{\qq}$ at~$\qq$ (cf.~\S1.8 of~\cite{1804.06400}).}
\end{remark}

\bibliographystyle{amsalpha}
\bibliography{snowden}

\providecommand{\bysame}{\leavevmode\hbox to3em{\hrulefill}\thinspace}
\providecommand{\MR}{\relax\ifhmode\unskip\space\fi MR }
\providecommand{\MRhref}[2]{%
  \href{http://www.ams.org/mathscinet-getitem?mr=#1}{#2}
}
\providecommand{\href}[2]{#2}
\begin{thebibliography}{WWE17}

\bibitem[AL78]{AL}
A.~O.~L. Atkin and Wen Ch'ing~Winnie Li, \emph{Twists of newforms and
  pseudo-eigenvalues of {$W$}-operators}, Invent. Math. \textbf{48} (1978),
  no.~3, 221--243. \MR{508986}

\bibitem[Buz14]{SmallBuzz}
Kevin Buzzard, \emph{Computing weight one modular forms over {$\Bbb C$} and
  {$\overline{\Bbb F}_p$}}, Computations with modular forms, Contrib. Math.
  Comput. Sci., vol.~6, Springer, Cham, 2014, pp.~129--146. \MR{3381451}

\bibitem[Cal18]{calegari2015non}
Frank Calegari, \emph{Non-minimal modularity lifting in weight one}, J. Reine
  Angew. Math. \textbf{740} (2018), 41--62. \MR{3824782}

\bibitem[Car94]{carayol1994formes}
Henri Carayol, \emph{Formes modulaires et repr\'esentations galoisiennes \`a
  valeurs dans un anneau local complet}, Contemporary Mathematics \textbf{165}
  (1994), 213--213.

\bibitem[CE98]{Coleman}
Robert~F. Coleman and Bas Edixhoven, \emph{On the semi-simplicity of the
  {$U_p$}-operator on modular forms}, Math. Ann. \textbf{310} (1998), no.~1,
  119--127. \MR{1600034}

\bibitem[CG18]{CG}
Frank Calegari and David Geraghty, \emph{Modularity lifting beyond the
  {T}aylor--{W}iles method}, Invent. Math. \textbf{211} (2018), no.~1,
  297--433. \MR{3742760}

\bibitem[Che14]{Chenevier}
Ga{\"e}tan Chenevier, \emph{The {$p$}-adic analytic space of pseudocharacters
  of a profinite group and pseudorepresentations over arbitrary rings},
  Automorphic forms and {G}alois representations. {V}ol. 1, London Math. Soc.
  Lecture Note Ser., vol. 414, Cambridge Univ. Press, Cambridge, 2014,
  pp.~221--285. \MR{3444227}

\bibitem[DDT97]{DDT}
Henri Darmon, Fred Diamond, and Richard Taylor, \emph{Fermat's last theorem},
  Elliptic curves, modular forms \& {F}ermat's last theorem ({H}ong {K}ong,
  1993), Int. Press, Cambridge, MA, 1997, pp.~2--140. \MR{1605752 (99d:11067b)}

\bibitem[Del71]{Deligne}
Pierre Deligne, \emph{Formes modulaires et repr\'esentations {$l$}-adiques},
  S\'eminaire {B}ourbaki. {V}ol. 1968/69: {E}xpos\'es 347--363, Lecture Notes
  in Math., vol. 175, Springer, Berlin, 1971, pp.~Exp.\ No.\ 355, 139--172.
  \MR{3077124}

\bibitem[DS74]{DS}
Pierre Deligne and Jean-Pierre Serre, \emph{Formes modulaires de poids {$1$}},
  Ann. Sci. \'Ecole Norm. Sup. (4) \textbf{7} (1974), 507--530 (1975).
  \MR{0379379}

\bibitem[Edi06]{Bas}
Bas Edixhoven, \emph{Comparison of integral structures on spaces of modular
  forms of weight two, and computation of spaces of forms mod 2 of weight one},
  J. Inst. Math. Jussieu \textbf{5} (2006), no.~1, 1--34, With appendix A (in
  French) by Jean-Fran{\c{c}}ois Mestre and appendix B by Gabor Wiese.
  \MR{2195943 (2007f:11046)}

\bibitem[Gro90]{gross1990tameness}
Benedict~H. Gross, \emph{A tameness criterion for galois representations
  associated to modular forms $(\mathrm{mod}\text{ }p)$}, Duke Mathematical
  Journal \textbf{61} (1990), no.~2, 445--517.

\bibitem[Kat73]{Katz}
Nicholas~M. Katz, \emph{{$p$}-adic properties of modular schemes and modular
  forms}, Modular functions of one variable, {III} ({P}roc. {I}nternat.
  {S}ummer {S}chool, {U}niv. {A}ntwerp, {A}ntwerp, 1972), Springer, Berlin,
  1973, pp.~69--190. Lecture Notes in Mathematics, Vol. 350. \MR{0447119}

\bibitem[Rou96]{MR1378546}
Rapha\"el Rouquier, \emph{Caract\'erisation des caract\`eres et
  pseudo-caract\`eres}, J. Algebra \textbf{180} (1996), no.~2, 571--586.
  \MR{1378546}

\bibitem[Sch15]{Schaeffer}
George~J. Schaeffer, \emph{Hecke stability and weight 1 modular forms}, Math.
  Z. \textbf{281} (2015), no.~1-2, 159--191. \MR{3384865}

\bibitem[Sno18]{Snowden}
Andrew Snowden, \emph{Singularities of ordinary deformation rings}, Math. Z.
  \textbf{288} (2018), no.~3-4, 759--781. \MR{3778977}

\bibitem[SW97]{SkinnerWiles}
C.~M. Skinner and A.~J. Wiles, \emph{Ordinary representations and modular
  forms}, Proc. Nat. Acad. Sci. U.S.A. \textbf{94} (1997), no.~20,
  10520--10527. \MR{1471466}

\bibitem[Tay91]{MR1115109}
Richard Taylor, \emph{Galois representations associated to {S}iegel modular
  forms of low weight}, Duke Math. J. \textbf{63} (1991), no.~2, 281--332.
  \MR{1115109}

\bibitem[WE18]{WE}
Carl Wang-Erickson, \emph{Algebraic families of {G}alois representations and
  potentially semi-stable pseudodeformation rings}, Math. Ann. \textbf{371}
  (2018), no.~3-4, 1615--1681. \MR{3831282}

\bibitem[Wie14]{Wiese}
Gabor Wiese, \emph{On {G}alois representations of weight one}, Doc. Math.
  \textbf{19} (2014), 689--707. \MR{3247800}

\bibitem[Wil88]{WilesOrdinary}
A.~Wiles, \emph{On ordinary {$\lambda$}-adic representations associated to
  modular forms}, Invent. Math. \textbf{94} (1988), no.~3, 529--573.
  \MR{969243}

\bibitem[WW18]{1804.06400}
Preston {Wake} and Carl {Wang-Erickson}, \emph{{The Eisenstein ideal with
  squarefree level}}, arXiv e-prints (2018), arXiv:1804.06400.

\bibitem[WWE17]{wake2015ordinary}
Preston Wake and Carl Wang-Erickson, \emph{Ordinary pseudorepresentations and
  modular forms}, Proc. Amer. Math. Soc. Ser. B \textbf{4} (2017), 53--71.
  \MR{3738092}

\end{thebibliography}

 \end{document}